\documentclass[reqno,oneside,12pt]{amsart}

\usepackage{hyperref}
\usepackage{longtable}
\usepackage[ansinew]{inputenc}
\usepackage{graphicx}
\usepackage{color}%
\usepackage[numbers, square]{natbib}
\usepackage{mathrsfs}
\usepackage{bbm}
\usepackage{tikz}
\usepackage{mathptmx}
\usepackage{amsmath,amsthm,amsfonts,amssymb}
\usepackage{dsfont,extarrows,enumerate}
\usepackage{fullpage}
\usepackage{xcolor}

\numberwithin{equation}{section}

\newcommand{\E}{\mathbb{E}}
\newcommand{\Q}{\mathbb{Q}}

\newcommand{\cZ}{C_{{\rm Sparse}}}
\newcommand{\cZS}{C_{{\rm Embed}}}
\newcommand{\tauE}{\tau_{{\rm Embed}}}
\newcommand{\tauS}{\tau_{{\rm Sparse}}}

\newcommand{\N}{\mathbb{N}}
\newcommand{\mmp}{\mathbb{P}}

\newcommand{\m}{\mathrm{m}}
\newcommand{\Var}{{\rm Var}\,}

\DeclareMathOperator{\1}{\mathbbm{1}}

\newtheorem{thm}{Theorem}[section]
\newtheorem{lemma}[thm]{Lemma}

\newtheorem{cor}[thm]{Corollary}

\newtheorem{assertion}[thm]{Proposition}
\theoremstyle{definition}

\theoremstyle{remark}
\newtheorem{rem}[thm]{Remark}

\begin{document}

\title{Critical branching processes in a sparse random environment}\date{}

\author{Dariusz Buraczewski}
\address{Dariusz Buraczewski, Mathematical Institute, University of Wroclaw, 50-384 Wroclaw, Po\-land}
\email{dariusz.buraczewski@math.uni.wroc.pl}

\author{Congzao Dong}
\address{Congzao Dong, School of Mathematics and Statistics, Xidian University, 710126 Xi'an, China}
\email{czdong@xidian.edu.cn}

\author{Alexander Iksanov}
\address{Alexander Iksanov, Faculty of Computer Science and Cybernetics, Taras Shev\-chen\-ko National University of Kyiv, 01601 Kyiv, Ukraine}
\email{iksan@univ.kiev.ua}

\author{Alexander Marynych}
\address{Alexander Marynych, Faculty of Computer Science and Cybernetics, Taras Shev\-chen\-ko National University of Kyiv, 01601 Kyiv, Ukraine}
\email{marynych@unicyb.kiev.ua}

\begin{abstract}
We introduce a branching process in a sparse random environment as an intermediate model between a Galton--Watson process and a branching process in a random environment. In the critical case we investigate 
the survival probability and prove Yaglom-type limit theorems,
that is, limit theorems for the size of population conditioned on the survival event. 
\end{abstract}



\maketitle

\section{Introduction and main results}

The branching process is a random process starting with one individual, the initial ancestor, which 
produces offspring according to some random rule. The collection
of offspring  constitutes the first generation. Each individual of the first generation gives birth to 
a random number of children with the same offspring 
distribution as for the initial ancestor. The numbers of offspring of different individuals (including the initial ancestor) are independent. 
This process continues forever or until the population dies out. An interesting problem is the behavior of the long-time evolution
of the process. Plainly, it depends on a particular rule that regulates giving birth to offspring.
In the simplest case, when the offspring distribution 
is the same for all generations, 
the branching process is called the Galton--Watson process. We refer to \cite{athreya:ney} for numerous results concerning, for instance, 
long-term survival or extinction of such a process, the growth rate of
the population, fluctuations of population sizes. Thanks to a simple 
tree structure, not only does the Galton--Watson process find 
numerous applications 
as a model of biological reproduction processes, but also in many other fields including 
computer science and physics.

The homogeneity of the Galton--Watson process 
reduces its applicability.
In some 
cases it may happen that the population evolution conditions change randomly over time.
We intend to study here branching processes in a randomly perturbed environment, in which homogeneity of the 
environment is modified on a sparse subset of $\N$.
To give a precise definition, let
$\mu$ be a fixed probability measure on $\N_0$ and 
$\Q = ((d_k,\nu_k))_{k\ge 1}$ 
a sequence of independent copies of a random vector $(d,\nu)$, where $d$ is a positive integer-valued random variable 
and $\nu$ is a random measure on $\N_0$ independent of $d$.
First we choose 
a subset of integers marked by the positions of a standard random walk $(S_k)_{k\geq 0}$ defined by 
$$
S_0=0,\quad S_k = \sum_{j=1}^k d_j,\quad k\in\mathbb{N},
$$
and then 
we impose random  measures at the marked 
sites.
The branching process in sparse random environment $\Q$ (BPSRE) is formally defined as follows:
$$
Z_0=1,\quad Z_{n+1} {=} \sum_{j=1}^{Z_n} \xi^{(n)}_j,\quad n\in\mathbb{N}_0:=\{0,1,2,\ldots\},
$$
where, if $n = S_k$ for some $k\in\mathbb{N}$, then, given $\Q$, $\xi^{(n)}_j$ are independent random variables 
with the common distribution $\nu_{k+1}$, which are also independent of $Z_n$.
Otherwise, if $n \notin \{S_0,S_1,S_2,\ldots\}$, then  $\xi^{(n)}_j$ are independent random variables with the common distribution $\mu$, which are also independent of $Z_n$. 
The process $(Z_n)_{n\ge 0}$ 
behaves like 
the Galton--Watson
process, with the exception of 
some randomly chosen generations in which 
the 
 offspring distribution 
is random.

The 
BPSRE that we are going to investigate 
here is an intermediate model between the branching process in random environment (BPRE) introduced by Smith and Wilkinson \cite{smith:wilkinson} and the Galton--Watson process.
The BPRE is a population growth process, in which the 
individuals reproduce independently of each other with the 
offspring distribution 
picked randomly at each generation. More precisely,
let $\nu$ be a random measure on the set of positive 
integers $\N$. Then a sequence  $(\nu_n )_{n \geq 1}$ of independent copies of $\nu$ can be interpreted as 
a random environment. The BPRE is then the sequence $Z'=(Z'_n)_{n \geq 0}$ defined by
the recursive formula 
$Z'_{n+1} {=} \sum_{k=1}^{Z'_n} \xi^{(n)}_k$,
where, given $(\nu_n )_{n \geq 1}$, $(\xi^{(n)}_k)_{k\geq 1}$ are independent identically distributed (iid) and independent of $Z'_n$ with the common distribution $\nu_{n+1}$.
We refer to the 
recent monograph by Kersting and Vatutin \cite{kersting2017discrete} for an overview of fundamental
properties of this process.

We intend to describe how the additional randomness of the environment affects 
the behavior of the BPSRE. To this end, 
we focus on Yaglom-type results. For the Galton--Watson process in the critical case, that is, 
when the expected number of offspring is $1$ (see (A2) below), it is known
that the probability of survival up to the generation $n$ is of the order $1/n$ and the population size conditioned to the survival set
converges weakly 
to an 
exponential distribution 
(section 9 in \cite{athreya:ney}). In contrast, in the critical case for the BPRE, that is, 
when the expectation of the logarithm of the number of offspring is $0$ (see (A1)), the probability of survival up to the generation $n$ is asymptotically $1/\sqrt n$,
and the process conditioned to the survival event converges weakly to a 
Rayleigh distribution. We prove below in Theorems \ref{thm:tail} and \ref{thm:flt} that, although the environment is random 
on a sparse subset only, the behavior the BPSRE reminds that 
of a BPRE.

To close the introduction, we 
mention that 
closely related 
random walks in a sparse random environment,
which is an intermediate model between the simple random walk and the random walk in a random environment,
have 
been recently investigated 
in~\cite{sparse:rwre:spa,sparse:rwre:ejp,buraczewski2023weak}.

\medskip

\subsection{Notation and assumptions}
Given a deterministic or random probability measure $\theta$ on $\mathbb{N}_0$, 
define the generating function
$$
f_{\theta}(s) = \sum_{j=0}^\infty s^j \theta(\{j\}), 
\quad |s|\leq 1.
$$
Denote by
$$A_{\theta}:= f_{\theta}^{\prime}(1) = \sum_{j=1}^\infty j \theta(\{j\})$$
its mean and by
$$
\sigma_{\theta}:= \frac{f''_{\theta}(1)}{(f'_{\theta}(1))^2} = \frac{1}{A^2_{\theta}}  \sum_{j=2 
}^\infty j(j-1)\theta(\{j\})
$$
its normalized second factorial moment. We shall also use a 
standardized truncated second moment defined by
$$
\kappa(f_{\theta};a):=\frac{1}{A^2_{\theta}}\sum_{j=a}^{\infty}j^2\theta(\{j\}),\quad a\in\mathbb{N}_0.
$$
To simplify our notation we shall write, for $k\geq 1$, $A_k$ and $\sigma_k$ instead of $A_{\nu_k}$ and $\sigma_{\nu_k}$, respectively. Thus, in our setting $(A_k)_{k\geq 1}$ and $(\sigma_k)_{k\geq 1}$ are two (dependent) sequences of iid 
random variables. As usual, $x^+=\max (x,0)$ and $x^-=\max (-x,0)$ for $x\in\mathbb{R}$.

Throughout the paper we impose the following assumptions:
\begin{enumerate}
\item[(A1)] $\E \log A_1=0$, $\mathfrak{v}^2:=\Var (\log A_1)\in (0,\infty)$ and $\E (\log^{-} A_1)^4<\infty$;
\item[(A2)] $A_{\mu}=1$;
\item[(A3)] $\E d_1^{3/2}<\infty$ and we put $\m:=\E d_1$;
\item[(A4)] $\E(\log^+\kappa(f_{\nu};a))^4<\infty$ for some $a\in\mathbb{N}$.
\end{enumerate}

\subsection{Main results}

Let $\tauS\in (0,\infty]$ be the extinction time of $(Z_n)_{n\geq 0}$, that is,
$$
\tauS:=\inf\{k\geq 0:Z_k=0\}.
$$
The following observation is almost immediate. 
\begin{assertion}\label{prop:criticallity}
Under the assumptions (A1)-(A2), 
$\mmp\{\tauS<\infty\}=1$.
\end{assertion}

Our first main result is concerned with the tail behavior of $\mathbb{P}\{\tauS>n\}=\mathbb{P}\{Z_n>0\}$ as $n\to\infty$.

\begin{thm}\label{thm:tail}
Assume (A1)-(A4). Then there exists $\cZ\in (0,\infty)$ such that
\begin{equation}\label{eq:tail-behavior}
\lim_{n\to\infty}\sqrt{n}\mathbb{P}\{\tauS>n\}=\cZ.
\end{equation}
\end{thm}

Our next result is a Yaglom-type functional limit theorem for the process $(Z_n)$. Recall that a Brownian meander, see~\cite{durrett1977weak}, is a stochastic process $(B_{+}(t))_{t\in [0,1]}$ defined as follows. Let $(B(t))_{t\in [0,1]}$ be a standard Brownian motion and $\zeta:=\sup\{t \in [0,1]:B(t)=0\}$ be its last visit to $0$ on $[0,1]$. Then
$$
B_{+}(t)=\frac{1}{\sqrt{1-\zeta}}|B(\zeta+t(1-\zeta))|,\quad t\in [0,1].
$$
\begin{thm}\label{thm:flt}
Assume (A1)-(A4). Then with $(B_{+}(t))_{t\in [0,1]}$ being a 
Brownian meander 
$$
{\rm Law}\left(\left(\frac{\log Z_{\lfloor nt\rfloor}}{\mathfrak{v}\sqrt{\m^{-1}n}}\right)_{t\in [0,1]}\,\Big|\, Z_n>0\right)~\Longrightarrow~{\rm Law}\left((B_{+}(t))_{t\in [0,1]}\right),\quad n\to\infty,
$$
weakly on the space of probability measures on $D[0,1]$ endowed with the Skorokhod $J_1$-topology.
\end{thm}
Using formula (1.1) in~\cite{durrett1977weak} we obtain 
the following one-dimensional result.
\begin{cor}\label{cor:clt}
Assume (A1)-(A4). Then, for every fixed $t\in (0,1]$,
\begin{equation}\label{eq:clt}
\lim_{n\to\infty}\mathbb{P}\left\{\frac{\log Z_{\lfloor nt\rfloor}}{\mathfrak{v}\sqrt{\m^{-1}n}}\geq x \,\Big|\, Z_{n}>0\right\}=\mathbb{P}\{B_{+}(t)\geq x\},\quad x\geq 0.
\end{equation}
The random variable 
$B_{+}(t)$ has an absolutely continuous distribution with a bounded nonvanishing density on $[0,\infty)$. Furthermore,
$$
\mathbb{P}\{B_{+}(1)\le x\}=1-e^{-x^2/2},\quad x\geq 0,
$$
so $B_{+}(1)$ has a Rayleigh distribution.
\end{cor}
\begin{rem}
The assumption (A4) and the last part of the assumption 
(A1) can be weakened without changing the formulations of the main results. A version of 
(A4) appears as Assumption (C) in~\cite[Chapter 5]{kersting2017discrete}. It is 
a convenient general condition allowing for an (asymptotically) closed form of the survival probability and also validity 
of a functional limit theorem for a critical branching process in iid random environment. A more general version of Assumption (C) can be found in~\cite{afanasyev:geiger:kersting:vatutin}. However, we prefer to sacrifice 
generality in favor 
of transparency and simplicity of the formulations.
\end{rem}

\section{Proofs}

The proof of our main results consists of three steps. First, we analyze an embedded process $(Z_{S_n})_{n\ge 0}$ by finding its survival asymptotic and proving a counterpart of 
Theorem~\ref{thm:flt}. Second, 
we deduce from the results obtained for $(Z_{S_n})_{n\ge 0}$ the corresponding statements 
for a randomly stopped process $(Z_{S_{\vartheta(n)}})_{n\ge 0}$, where $(\vartheta(n))_{n\geq 0}$ is the first passage time process for the random walk $(S_k)_{k\geq 0}$. At 
the last step, we show that $(Z_{S_{\vartheta(n)}})_{n\ge 0}$ is uniformly close to $(Z_{n})_{n\ge 0}$.

\subsection{Analysis of the embedded process}

Observe that $(Z_{S_k})_{k\ge 0}$ is a branching process in iid random environment $\widetilde{\mathbb{Q}}=(\widetilde{\nu}_k)_{k\geq 1}$ which can be explicitly described as follows. Let $((\widetilde{Z}^{(i)}_j)_{j\geq 0})_{i\geq 0}$ 
be a sequence 
of 
independent 
copies of a critical Galton--Watson process $(\widetilde{Z}_j)_{j\geq 0}$ in deterministic environment with the offspring distribution $\mu$ and $\widetilde{Z}_0=1$. Suppose that $((\widetilde{Z}^{(i)}_j)_{j\geq 0})_{i\geq 0}$ is independent of the environment $\mathbb{Q}$. Then
\begin{equation}\label{eq:embedded_offspring_distribution}
\widetilde{\nu}_k(\{j\})=\sum_{l=0}^{\infty}\nu_k(\{l\})\mathbb{P}\left\{\sum_{i=1}^{l}\widetilde{Z}^{(i)}_{d_k-1}=j\right\},\quad k,j\in\mathbb{N}_0.
\end{equation}
Let $\widetilde{\nu}$ be a generic copy of iid random measures $(\widetilde{\nu}_k)_{k\ge 1}$. Put
$$
\widetilde{g}(s):=\E s^{\widetilde{Z}_{d-1}},\quad |s|\leq 1,
$$
where $d$ is assumed independent of $(\widetilde{Z}_k)_{k\geq 0}$. Equality 
~\eqref{eq:embedded_offspring_distribution} entails 
that the 
generating function of the random measure $\widetilde{\nu}$ is given by
$$
f_{\widetilde{\nu}}(s)=f_{\nu}(\widetilde{g}(s)),\quad |s|\leq 1.
$$
Since $\widetilde{g}^{\prime}(1)=\E\widetilde{Z}_{d-1}=1$, the latter 
formula immediately implies that
\begin{equation}\label{eq:a_embedded}
A_{\widetilde{\nu}_k}=f^{\prime}_{\widetilde{\nu}_k}(1)=f^{\prime}_{\nu_k}(1)=A_{\nu_k}=A_k,\quad k\in\mathbb{N}_0.
\end{equation}
Further,
$$
\sigma_{\widetilde{\nu}_k}=\frac{f^{\prime\prime}_{\widetilde{\nu}_k}(1)}{(f^{\prime}_{\widetilde{\nu}_k}(1))^2}=\frac{f^{\prime\prime}_{\nu_k}(1)+f^{\prime}_{\nu_k}(1)\widetilde{g}^{\prime\prime}(1)}{(f^{\prime}_{\nu_k}(1))^2}=\sigma_{\nu_k}+\frac{\sigma_{\mu}(\E d-1)}{A_{\nu_k}}=\sigma_k+\frac{\sigma_{\mu}(\E d-1)}{A_k},\quad k\in\mathbb{N}_0,
$$
where we have used that
\begin{equation}\label{eq:g_tilde_second_moment}
\widetilde{g}^{\prime\prime}(1)=\E \widetilde{Z}_{d-1}(\widetilde{Z}_{d-1}-1)=\sigma_{\mu}(\E d-1),
\end{equation} see, for instance, 
Chapter I.2 in \cite{athreya:ney} for the last equality.
Note that~\eqref{eq:a_embedded} guarantees 
that
$$
\E \log A_{\widetilde{\nu}_1} = \E \log A_{\nu_1} = \E \log A_1 = 0,
$$
which means that the embedded process $(Z_{S_n})_{n\geq 0}$ is {\bf critical}. In particular,
\begin{equation}\label{eq:tau_S_def}
\tauE:=\inf\{k\geq 0:Z_{S_k}=0\}<\infty\quad\text{a.s.}
\end{equation}

Recall that 
we denote by $\kappa(f_{\theta};a)$ the truncated second moment of a measure $\theta$.
\begin{lemma}
Let $a_{\ast}\in\mathbb{N}_0$ and assume that 
$\E (\log^{+}\kappa(f_{\nu};a_{\ast}))^4<\infty$ and $\E (\log^{-} A_{\nu})^4<\infty$. Then $\E (\log^{+}\kappa(f_{\widetilde{\nu}};a_{\ast}))^4<\infty$.
\end{lemma}
\begin{proof}
We start by writing 
\begin{multline*}
\kappa(f_{\widetilde{\nu}};a)=\frac{1}{A^2_{\widetilde{\nu}}}\sum_{j=a}^{\infty}j^2\sum_{l=0}^{\infty}\nu(\{l\})\mathbb{P}\left\{\sum_{i=1}^{l}\widetilde{Z}^{(i)}_{d-1}=j\right\}\\
=\frac{1}{A_{\nu}^2}\sum_{l=0}^{\infty}\nu(\{l\}) \E\left(\left(\sum_{i=1}^{l}\widetilde{Z}^{(i)}_{d-1}\right)^2 \1_{\left\{\sum_{i=1}^{l}\widetilde{Z}^{(i)}_{d-1}\geq a\right\}}\right).
\end{multline*}
In view of~\eqref{eq:g_tilde_second_moment}, for all $a\in\mathbb{N}$,
\begin{align*}
\E\left(\left(\sum_{i=1}^{l}\widetilde{Z}^{(i)}_{d-1}\right)^2 \1_{\left\{\sum_{i=1}^{l}\widetilde{Z}^{(i)}_{d-1}\geq a\right\}}\right)\leq \E\left(\sum_{i=1}^{l}\widetilde{Z}^{(i)}_{d-1}\right)^2 \leq C_1 m l^2,
\end{align*}
where 
$m = \E d$ and 
$C_1>0$ is a constant. Thus, it suffices to check that
$$
\E \left(\log^{+} \frac{1}{A_{\nu}^2}\sum_{l=0}^{a_{\ast}}\nu(\{l\}) \E\left(\left(\sum_{i=1}^{l}\widetilde{Z}^{(i)}_{d-1}\right)^2 \1_{\left\{\sum_{i=1}^{l}\widetilde{Z}^{(i)}_{d-1}\geq a_{\ast}\right\}}\right)\right)^4<\infty.
$$
The inner expectation is equal to $0$ if $l=0$ and is uniformly bounded by a constant $C_2>0$ for all $l=1,\ldots,a_{\ast}$. It remains to note that
\begin{multline*}
\E \left(\log^{+} \frac{C_2}{A_{\nu}^2}\sum_{l=1}^{a_{\ast}}\nu(\{l\})\right)^4\leq \E \left(\log^{+} \frac{C_2}{A_{\nu}^2}\sum_{l=1}^{\infty}l\nu(\{l\})\right)^4\leq \E \left(\log^{+} \frac{C_2}{A_{\nu}}\right)^4\\
\leq C_3 \E \left(\log^{+} \frac{1}{A_{\nu}}\right)^4+C_4 =C_3 \E \left(\log^{-} A_{\nu}\right)^4+C_4<\infty
\end{multline*}
for some $C_3>0$ and $C_4\geq 0$.
\end{proof}

Using~Theorem 5.1 on p.~107 in \cite{kersting2017discrete} 
we obtain the following result.
\begin{assertion}\label{prop:embedded_process_survival}
Assume (A1), (A2), (A4) and $\E d<\infty$.
Then
$$
\mathbb{P}\{Z_{S_n}>0\}~\sim~\frac{\cZS}{\sqrt{n}},\quad n\to\infty
$$
for some constant $\cZS>0$.
\end{assertion}

Furthermore, Theorem 5.6 on p.~126 in \cite{kersting2017discrete} entails the proposition. 
\begin{assertion}\label{prop:embedded_process_flt}
Assume (A1), (A2), (A4) and $\E d<\infty$ and $\E(\kappa(f_{\nu};a))^4<\infty$ for some $a\in\mathbb{N}_0$. Then, with $(B_{+}(t))_{t\in [0,1]}$ being the Brownian meander, 
$$
{\rm Law}\left(\left(\frac{\log Z_{S_{\lfloor nt\rfloor}}}{\mathfrak{v}\sqrt{n}}\right)_{t\in [0,1]}\,\Big|\, Z_{S_n}>0\right)~\Longrightarrow~{\rm Law}\left((B_{+}(t))_{t\in [0,1]}\right),\quad n\to\infty
$$
weakly on the space of probability measures on $D[0,1]$ endowed with the Skorokhod $J_1$-topology.
\end{assertion}
Given next is the corollary which 
follows from formula 
(1.1) in~\cite{durrett1977weak}.
\begin{cor}\label{cor:embedded_process_clt}
Under the assumptions of Proposition~\ref{prop:embedded_process_flt}, for every fixed $t\in (0,1]$,
\begin{equation}\label{eq:embedded_process_clt}
\lim_{n\to\infty}\mathbb{P}\left\{\frac{\log Z_{S_{\lfloor nt\rfloor}}}{\mathfrak{v}\sqrt{n}}\geq x \,\Big|\, Z_{S_n}>0\right\}=\mathbb{P}\{B_{+}(t)\geq x\},\quad x\geq 0.
\end{equation}
The random variable 
$B_{+}(t)$ has an absolutely continuous distribution with a bounded nonvanishing density on $[0,\infty)$.
\end{cor}
Propositions~\ref{prop:embedded_process_survival} and~\ref{prop:embedded_process_flt} are the key ingredients for the proof of our main results.

\subsection{Proof of Proposition~\ref{prop:criticallity} and Theorem~\ref{thm:tail}}
Recall that $\tauE=\inf\{k\geq 0:Z_{S_k}=0\}$ is the extinction time of the embedded process $(Z_{S_k})_{k\geq 0}$ and note that
$$
\mathbb{P}\{\tauS<\infty\}\geq \mathbb{P}\{\tauE<\infty\}=1,
$$
where the equality is justified by ~\eqref{eq:tau_S_def}. This proves Proposition~\ref{prop:criticallity}.

For $n\in\mathbb{N}_0$, define the first passage time $\vartheta(n)$ by
\begin{equation}\label{eq:first_passage_time}
\vartheta(n):=\inf\{k\geq 0:S_k > n\}.
\end{equation}
Note that
$$
\mathbb{P}\{Z_{S_{\vartheta(n)}}>0\}\leq \mathbb{P}\{Z_n>0\}\leq \mathbb{P}\{Z_{S_{\vartheta(n)-1}}>0\},\quad n\in\mathbb{N}_0.
$$
In view of the strong law of large numbers for $\vartheta(n)$, which reads 
$$
\frac{\vartheta(n)}{n}~\to~\frac{1}{{\tt m}},\quad n\to\infty,\quad\text{a.s.},
$$
and Proposition~\ref{prop:embedded_process_survival}, it is natural to expect that
\begin{equation}\label{eq:equivalence_two_sided_estimate}
\mathbb{P}\{Z_{S_{\vartheta(n)}}>0\}~\sim~\frac{{\tt m}^{1/2}\cZS}{\sqrt{n}}~\sim~\mathbb{P}\{Z_{S_{\vartheta(n)-1}}>0\},\quad n\to\infty.
\end{equation}
Checking relation~\eqref{eq:equivalence_two_sided_estimate} is clearly sufficient for a proof of Theorem~\ref{thm:tail}. Furthermore, \eqref{eq:equivalence_two_sided_estimate} would demonstrate 
that
\begin{equation}\label{eq:equa}
\cZ=\m^{1/2}\cZS.
\end{equation}

Observe that
$$
\mathbb{P}\{Z_{S_{\vartheta(n)}}>0\}=\mathbb{P}\{\tauE>\vartheta(n)\}=\mathbb{P}\{\tauE-1\geq \vartheta(n)\}=\mathbb{P}\{S_{\tauE-1}>n\},
$$
and, similarly,
$$
\mathbb{P}\{Z_{S_{\vartheta(n)-1}}>0\}=\mathbb{P}\{\tauE>\vartheta(n)-1\}=\mathbb{P}\{\tauE\geq \vartheta(n)\}=\mathbb{P}\{S_{\tauE}>n\}.
$$
The desired relation~\eqref{eq:equivalence_two_sided_estimate} follows from Theorem 3.1 in~\cite{robert2008tails} applied with $r=3/2$ provided we can check that
$$
n\mathbb{P}\{d>n\}=o(\mathbb{P}\{\tauE>n\})=o(\mathbb{P}\{Z_{S_n}>0\}),\quad n\to\infty.
$$
By Proposition ~\ref{prop:embedded_process_survival}, this  
is equivalent 
to
$$
\mathbb{P}\{d>n\}=o(n^{-3/2}),\quad n\to\infty,
$$
which 
is secured by assumption (A3). This completes the proof of Theorem~\ref{thm:tail}.

\subsection{Proof of Theorem~\ref{thm:flt}}
We start by noting that $\m<\infty$ together with the strong law of large numbers for $(\vartheta(n))$ imply
$$
\sup_{t\in [0,1]}\left|\frac{\vartheta(\lfloor nt\rfloor)-1}{n}-\frac{t}{\m}\right|~\to 0, 
\quad n\to\infty \quad \mbox{a.s.}
$$
Thus, the weak convergence claimed in Proposition~\ref{prop:embedded_process_flt} can be strengthened 
to the joint convergence
\begin{multline*}
{\rm Law}\left(\left(\left(\frac{\log Z_{S_{\lfloor \m^{-1}nt\rfloor}}}{\mathfrak{v}\sqrt{\m^{-1} n}},\frac{\vartheta(\lfloor nt\rfloor)-1}{\m^{-1}n}\right)_{t\in [0,1]}\,\Big|\, Z_{S_{\lfloor \m^{-1}n\rfloor}}>0\right)\right)\\
~\Longrightarrow~{\rm Law}\left(\left(B_{+}(t),t\right)_{t\in [0,1]}\right),\quad n\to\infty,
\end{multline*}
which holds weakly on the space of probability measures on $D[0,1]\times D[0,1]$ endowed with the product 
$J_1$-topology. 
Using 
the continuous mapping theorem in combination with continuity of the composition (see, for instance, Theorem 13.2.2 in \cite{whitt2002stochastic}) 
we infer 
\begin{equation}\label{eq:flt_proof1}
{\rm Law}\left(\left(\left(\frac{\log^{+} Z_{S_{\vartheta(\lfloor nt\rfloor)-1}}}{\mathfrak{v}\sqrt{\m^{-1}n}}\right)_{t\in [0,1]}\,\Big|\, Z_{S_{\lfloor \m^{-1}n\rfloor}}>0\right)\right)~\Longrightarrow~{\rm Law}\left((B_{+}(t))_{t\in [0,1]}\right),\quad n\to\infty
\end{equation}
weakly on the space of probability measures on $D[0,1]$. We have replaced $\log$ by $\log^{+}$ in~\eqref{eq:flt_proof1} because 
the event $\{Z_{S_{\lfloor \m^{-1}n\rfloor}}>0\}$ does not entail 
the event
$$
\Big\{Z_{S_{\vartheta(\lfloor nt\rfloor)-1}}>0\text{ for all } t\in [0,1]\Big\}.
$$
Now we 
check that~\eqref{eq:flt_proof1} secures 
\begin{equation}\label{eq:flt_proof2}
{\rm Law}\left(\left(\left(\frac{\log Z_{S_{\vartheta(\lfloor nt\rfloor)-1}}}{\mathfrak{v}\sqrt{\m^{-1}n}}\right)_{t\in [0,1]}\,\Big|\, Z_{n}>0\right)\right)~\Longrightarrow~{\rm Law}\left((B_{+}(t))_{t\in [0,1]}\right),\quad n\to\infty.
\end{equation}
By Proposition~\ref{prop:embedded_process_survival}, 
Theorem~\ref{thm:tail} and \eqref{eq:equa},
\begin{equation}\label{eq:flt_proof21}
\mathbb{P}\{Z_{n}>0\}~\sim~\mathbb{P}\{Z_{S_{\lfloor \m^{-1}n\rfloor}}>0\}~\sim~ \frac{\cZ}{\sqrt{n}},\quad n\to\infty.
\end{equation}
Thus, limit relation~\eqref{eq:flt_proof2} follows once we can prove that
\begin{equation}\label{eq:flt_proof3}
\lim_{n\to\infty}\sqrt{n}\mathbb{P}\{Z_{S_{\lfloor \m^{-1}n\rfloor}}>0,Z_n=0\}=\lim_{n\to\infty}\sqrt{n}\mathbb{P}\{Z_{S_{\lfloor \m^{-1}n\rfloor}}=0,Z_n>0\}=0.
\end{equation}
In view of~\eqref{eq:flt_proof21}, it suffices to show 
that
$$
\lim_{n\to\infty}\mathbb{P}\{Z_n>0\,|\,Z_{S_{\lfloor \m^{-1}n\rfloor}}>0\}=1.
$$
Fix any $\varepsilon>0$. The assumption (A3) implies that 
$$
\mathbb{P}\{|S_n-\m n|\geq \varepsilon n\}=o(n^{-1/2}),\quad n\to\infty,
$$
by Theorem 4 in~\cite{baum1965convergence}.
Thus, 
\begin{align*}
\mathbb{P}\{Z_n>0\,|\,Z_{S_{\lfloor \m^{-1}n\rfloor}}>0\}&=\mathbb{P}\{Z_n>0, S_{\lfloor \m^{-1}(1+\varepsilon)n\rfloor}>n\,|\,Z_{S_{\lfloor \m^{-1}n\rfloor}}>0\}+o(1)\\
&\geq \mathbb{P}\{Z_{S_{\lfloor \m^{-1}(1+\varepsilon)n\rfloor}}>0, S_{\lfloor \m^{-1}(1+\varepsilon)n\rfloor}>n\,|\,Z_{S_{\lfloor \m^{-1}n\rfloor}}>0\}+o(1)\\
&=\mathbb{P}\{Z_{S_{\lfloor \m^{-1}(1+\varepsilon)n\rfloor}}>0\,|\,Z_{S_{\lfloor \m^{-1}n\rfloor}}>0\}+o(1)\\
&=\frac{\mathbb{P}\{Z_{S_{\lfloor \m^{-1}(1+\varepsilon)n\rfloor}}>0\}}{\mathbb{P}\{Z_{S_{\lfloor \m^{-1}n\rfloor}}>0\}}+o(1)\\
&\to (1+\varepsilon)^{-1/2},\quad n\to\infty,
\end{align*}
where we have used Proposition~\ref{prop:embedded_process_survival} for the last passage. Sending $\varepsilon\to 0$ gives~\eqref{eq:flt_proof3}. 

To finish the proof of Theorem~\ref{thm:flt} it remains to check that, for all 
$\varepsilon>0$,
\begin{equation}\label{eq:flt_proof5}
\lim_{n\to\infty}\mathbb{P}\left\{ \sup_{t\in [0,1]}\left|\frac{\log Z_{\lfloor nt\rfloor}-\log Z_{S_{\vartheta(\lfloor nt\rfloor)-1}}}{\mathfrak{v}\sqrt{\m^{-1}n}}\right|>\varepsilon \,\Big|\, Z_n>0\right\}=0.
\end{equation}
To this end, we need an auxiliary lemma. 

\begin{lemma}\label{lem1}
Assume (A2),  $\E d<\infty$ and that $d$ is independent of $(\widetilde{Z}_j)_{j\geq 0}$. Then $$
\E \left(\max_{0\leq k\leq d} \widetilde{Z}_k\right)\leq 1+\E d<\infty.
$$
\end{lemma}
\begin{proof}
The proof follows from the chain of inequalities
$$
\E \left(\max_{0\leq k\leq d} \widetilde{Z}_k\right) \leq
\mathbb{E} \left(\sum_{k\geq 0}\widetilde{Z}_k\1_{\{d\geq k\}}\right)=\sum_{k\geq 0}\E\widetilde{Z}_k\cdot\mathbb{P}\{d\geq k\}=1+\E d.
$$
\end{proof}
In order to prove~\eqref{eq:flt_proof5} we first show that 
\begin{equation}\label{eq:flt_proof51}
\lim_{n\to\infty}\mathbb{P}\left\{ \sup_{t\in [0,1]}\frac{\log Z_{\lfloor nt\rfloor}-\log Z_{S_{\vartheta(\lfloor nt\rfloor)-1}}}{\mathfrak{v}\sqrt{\m^{-1}n}}>\varepsilon \,\Big|\, Z_n>0\right\}=0.
\end{equation}
Note that
\begin{equation}\label{eq:decomposition}
Z_{\lfloor nt\rfloor}=\sum_{j=1}^{Z_{S_{\vartheta(\lfloor nt\rfloor)-1}}}\widetilde{Z}^{(j)}_{\lfloor nt\rfloor-S_{\vartheta(\lfloor nt\rfloor)-1}}(S_{\vartheta(\lfloor nt\rfloor)-1}),\quad t\in [0,1],\quad n\in\mathbb{N},
\end{equation}
where $(\widetilde{Z}^{(j)}_k(m))_{k\geq 0}$ is the Galton--Watson process initiated by the $j$-th individual in the generation $m$. 
On the event $\{Z_n>0\}$,
$$
\frac{Z_{\lfloor nt\rfloor}}{Z_{S_{\vartheta(\lfloor nt\rfloor)-1}}}=\frac{\sum_{j=1}^{Z_{S_{\vartheta(\lfloor nt\rfloor)-1}}}\widetilde{Z}^{(j)}_{\lfloor nt\rfloor-S_{\vartheta(\lfloor nt\rfloor)-1}}(S_{\vartheta(\lfloor nt\rfloor)-1})}{Z_{S_{\vartheta(\lfloor nt\rfloor)-1}}},\quad t\in [0,1],\quad n\in\mathbb{N},
$$
and thereupon
\begin{equation*}
\sup_{t\in [0,1]}\frac{Z_{\lfloor nt\rfloor}}{Z_{S_{\vartheta(\lfloor nt\rfloor)-1}}}\leq \sup_{1\leq k\leq \vartheta(n)}\frac{\sum_{j=1}^{Z_{S_{k-1}}}\max_{0\leq i\leq d_k}\widetilde{Z}^{(j)}_{i}(S_{k-1})}{Z_{S_{k-1}}},\quad n\in\mathbb{N}.
\end{equation*}
Instead of~\eqref{eq:flt_proof51}, we shall prove 
a stronger relation
\begin{equation}\label{eq:flt_proof511}
\lim_{n\to\infty}\mathbb{P}\left\{ \sum_{1\leq k\leq \vartheta(n)}\frac{\sum_{j=1}^{Z_{S_{k-1}}}\max_{0\leq i\leq d_k}\widetilde{Z}^{(j)}_{i}(S_{k-1})}{Z_{S_{k-1}}}>\varepsilon n^3 \,\Big|\, Z_n>0\right\}=0.
\end{equation}
By Markov's inequality in combination with 
$\mathbb{P}\{Z_n>0\}\geq (1/C_5) 
n^{-1/2}$ for some $C_5>0$ and large $n$. 
\begin{align*}
&\hspace{-1cm}\mathbb{P}\left\{ \sum_{1\leq k\leq \vartheta(n)}\frac{\sum_{j=1}^{Z_{S_{k-1}}}\max_{0\leq i\leq d_k}\widetilde{Z}^{(j)}_{i}(S_{k-1})}{Z_{S_{k-1}}}>\varepsilon n^3 \,\Big|\, Z_n>0\right\}\\
&\leq \varepsilon^{-1} n^{-3}\sum_{k=1}^{\infty}\E \left( \1_{\{S_{k-1}\leq n\}}\frac{1}{Z_{S_{k-1}}}\sum_{j=1}^{Z_{S_{k-1}}}\max_{0\leq i\leq d_k}\widetilde{Z}^{(j)}_{i}(S_{k-1}) \,\Big|\, Z_{n}>0\right)\\
&\leq C_5 \varepsilon^{-1} n^{-5/2}\sum_{k=1}^{\infty}\E \left( \1_{\{S_{k-1}\leq n,\,Z_n>0\}}\frac{1}{Z_{S_{k-1}}}\sum_{j=1}^{Z_{S_{k-1}}}\max_{0\leq i\leq d_k}\widetilde{Z}^{(j)}_{i}(S_{k-1})\right)\\
&\leq C_5 \varepsilon^{-1} n^{-5/2}\sum_{k=1}^{\infty}\E \left( \1_{\{S_{k-1}\leq n,\,Z_{S_{k-1}}>0\}}\frac{1}{Z_{S_{k-1}}}\sum_{j=1}^{Z_{S_{k-1}}}\max_{0\leq i\leq d_k}\widetilde{Z}^{(j)}_{i}(S_{k-1})\right)\\
&=C_5 \varepsilon^{-1} n^{-5/2}\left(\E\max_{0\leq i\leq d}\widetilde{Z}_{i}\right)\sum_{k=1}^{\infty} \mathbb{P}\{S_{k-1}\leq n\}=O(n^{-3/2}),\quad n\to\infty.
\end{align*}
To justify the penultimate equality, observe that, given $(Z_{S_{k-1}},S_{k-1})$, the sequences 
$$
(\widetilde{Z}^{(1)}_i(S_{k-1}))_{i\geq 0},\ldots,(\widetilde{Z}^{(Z_{S_{k-1}})}_i(S_{k-1}))_{i\geq 0}
$$
are independent copies of the critical Galton--Watson process $(\widetilde{Z}_i)_{i\geq 0}$. The last equality is a consequence of 
Lemma~\ref{lem1} and the elementary renewal theorem which states that 
$$
\sum_{k=1}^{\infty} \mathbb{P}\{S_{k-1}\leq n\}=\E\vartheta (n)~\sim~\frac{n}{\m},\quad n\to\infty.
$$

We shall now check that
\begin{equation}\label{eq:flt_proof52}
\lim_{n\to\infty}\mathbb{P}\left\{ \inf_{t\in [0,1]}\frac{\log Z_{\lfloor nt\rfloor}-\log Z_{S_{\vartheta(\lfloor nt\rfloor)-1}}}{\mathfrak{v}\sqrt{\m^{-1}n}}<-\varepsilon \,\Big|\, Z_n>0\right\}=0.
\end{equation}
Using again decomposition~\eqref{eq:decomposition}, we write on the event $\{Z_n>0\}$
\begin{align*}
\inf_{t\in [0,1]}\frac{Z_{\lfloor nt\rfloor}}{Z_{S_{\vartheta(\lfloor nt\rfloor)-1}}}&=\inf_{t\in [0,1]}\frac{\sum_{j=1}^{Z_{S_{\vartheta(\lfloor nt\rfloor)-1}}}\widetilde{Z}^{(j)}_{\lfloor nt\rfloor-S_{\vartheta(\lfloor nt\rfloor)-1}}(S_{\vartheta(\lfloor nt\rfloor)-1})}{Z_{S_{\vartheta(\lfloor nt\rfloor)-1}}}\\&\geq \inf_{1\leq k\leq \vartheta(n)}\frac{\sum_{j=1}^{Z_{S_{k-1}}}\min_{0\leq i\leq d_k}\widetilde{Z}^{(j)}_i(S_{k-1})}{Z_{S_{k-1}}}\\
&\geq\inf_{1\leq k\leq \vartheta(n)}\frac{\sum_{j=1}^{Z_{S_{k-1}}}\1_{\{\widetilde{Z}^{(j)}_{d_k}(S_{k-1})>0\}}}{Z_{S_{k-1}}},\quad n\in\mathbb{N},
\end{align*}
As in the proof of~\eqref{eq:flt_proof51}, we shall prove a relation which is stronger 
than~\eqref{eq:flt_proof52}, namely,
\begin{equation}\label{eq:flt_proof521}
\lim_{n\to\infty}\mathbb{P}\left\{\inf_{1\leq k\leq \vartheta(n)}\frac{\sum_{j=1}^{Z_{S_{k-1}}}\1_{\{\widetilde{Z}^{(j)}_{d_k}(S_{k-1})>0\}}}{Z_{S_{k-1}}} < \varepsilon n^{-3} \,\Big|\, Z_n>0\right\}=0.
\end{equation}
Since $\mathbb{P}\{Z_{S_{\vartheta(n)}}>0\}\sim \mathbb{P}\{Z_n>0\}$ as $n\to\infty$, by 
~\eqref{eq:equivalence_two_sided_estimate}, and $\{Z_{S_{\vartheta(n)}}>0\}$ entails 
$\{Z_n>0\}$, relation~\eqref{eq:flt_proof521} is equivalent to
\begin{equation}\label{eq:flt_proof5211}
\lim_{n\to\infty}\mathbb{P}\left\{\inf_{1\leq k\leq \vartheta(n)}\frac{\sum_{j=1}^{Z_{S_{k-1}}}\1_{\{\widetilde{Z}^{(j)}_{d_k}(S_{k-1})>0\}}}{Z_{S_{k-1}}} < \varepsilon n^{-3} \,\Big|\, Z_{S_{\vartheta(n)}}>0\right\}=0.
\end{equation}
Observe 
that on the event $\{Z_{S_{\vartheta(n)}}>0\}$ 
$$
\sum_{j=1}^{Z_{S_{k-1}}}\1_{\{\widetilde{Z}^{(j)}_{d_k}(S_{k-1})>0\}}>0,\quad k\leq \vartheta(n),
$$
since otherwise the population does not survive up to time $S_{\vartheta(n)}$. Using this and 
the union bound yields
\begin{multline*}
\mathbb{P}\left\{\inf_{1\leq k\leq \vartheta(n)}\frac{1}{Z_{S_{k-1}}}\sum_{j=1}^{Z_{S_{k-1}}}\1_{\{\widetilde{Z}^{(j)}_{d_k}(S_{k-1})>0\}} < \varepsilon n^{-3} \,\Big|\, Z_{S_{\vartheta(n)}}>0\right\}\\
\leq \sum_{k\geq 1}\mathbb{P}\left\{0<\frac{1}{Z_{S_{k-1}}}\sum_{j=1}^{Z_{S_{k-1}}}\1_{\{\widetilde{Z}^{(j)}_{d_k}(S_{k-1})>0\}} < \varepsilon n^{-3},k\leq \vartheta(n) \,\Big|\, Z_{S_{\vartheta(n)}}>0\right\}.
\end{multline*}
Invoking 
$\mathbb{P}\{Z_{S_{\vartheta(n)}}>0\}\geq (1/C_6) 
n^{-1/2}$ for some $C_6>0$ and large $n$, we obtain, for such $n$,
\begin{align}
&\hspace{-1cm}\sum_{k\geq 1}\mathbb{P}\left\{0<\frac{1}{Z_{S_{k-1}}}\sum_{j=1}^{Z_{S_{k-1}}}\1_{\{\widetilde{Z}^{(j)}_{d_k}(S_{k-1})>0\}} < \varepsilon n^{-3},k\leq \vartheta(n) \,\Big|\, Z_{S_{\vartheta(n)}}>0\right\}\notag\\
&\leq C_6 
n^{1/2}\sum_{k\geq 1}\mathbb{P}\left\{0<\frac{1}{Z_{S_{k-1}}}\sum_{j=1}^{Z_{S_{k-1}}}\1_{\{\widetilde{Z}^{(j)}_{d_k}(S_{k-1})>0\}} < \varepsilon n^{-3},k\leq \vartheta(n), Z_{S_{\vartheta(n)}}>0\right\}\notag\\
&= C_6 
n^{1/2}\sum_{k\geq 1}\mathbb{P}\left\{0<\frac{1}{Z_{S_{k-1}}}\sum_{j=1}^{Z_{S_{k-1}}}\1_{\{\widetilde{Z}^{(j)}_{d_k}(S_{k-1})>0\}} < \varepsilon n^{-3},S_{k-1}\leq n, Z_{S_{\vartheta(n)}}>0\right\}\notag\\
&\leq C_6 
n^{1/2}\sum_{k\geq 1}\mathbb{P}\left\{0<\frac{1}{Z_{S_{k-1}}}\sum_{j=1}^{Z_{S_{k-1}}}\1_{\{\widetilde{Z}^{(j)}_{d_k}(S_{k-1})>0\}} < \varepsilon n^{-3},S_{k-1}\leq n, Z_{S_{k-1}}>0\right\}\label{eq:flt_proof52111}.
\end{align}
Let $\widetilde{p}:=\mathbb{P}\{\widetilde{Z}_{d}>0\}$ be the probability of the event that the critical Galton--Watson process $(\widetilde{Z}_k)_{k\geq 0}$ survives up to random time $d$ independent of $(\widetilde{Z}_k)_{k\geq 0}$. Obviously,
$\widetilde{p}\in (0,1)$. Given $(S_{k-1},Z_{S_{k-1}})$, the sum
$$
\sum_{j=1}^{Z_{S_{k-1}}}\1_{\{\widetilde{Z}^{(j)}_{d_k}(S_{k-1})>0\}}
$$
has a binomial distribution with parameters $(Z_{S_{k-1}},\widetilde{p})$.

In what follows 
we denote by ${\rm Bin}(N,p)$ a random variable with a binomial distribution with $N$ and $p$ interpreted as 
the number of independent trials and 
a success probability, respectively. 
The next lemma provides a uniform in $N$ estimate for $\mathbb{P}\{0 < N^{-1}{\rm Bin}(N,p)\leq x\}$, which is useful when $x$ is close to zero.

\begin{lemma}\label{lem:1}
For all $N\in\mathbb{N}$ and $x\in (0,p)$, 
$$
\mathbb{P}\{0<N^{-1}{\rm Bin}(N,p)\leq x\}\leq \frac{p(1-p)x}{(p-x)^2}. 
$$
\end{lemma}
\begin{proof}
Plainly, 
$\mathbb{P}\{0<N^{-1}{\rm Bin}(N,p)\leq x\}=0$ if $x<1/N$. If $x\geq 1/N$, then by Chebyshev's inequality
\begin{multline*}
\mathbb{P}\{0<N^{-1}{\rm Bin}(N,p)\leq x\}\leq \mathbb{P}\{{\rm Bin}(N,p)\leq Nx\}\\
= \mathbb{P}\{{\rm Bin}(N,1-p)-N(1-p)\geq N(p-x)\}
\leq 
\frac{p(1-p)}{(p-x)^2}\frac{1}{N}\leq \frac{p(1-p)}{(p-x)^2}x.
\end{multline*}
\end{proof}
Using Lemma \ref{lem:1} 
we estimate the summands in~\eqref{eq:flt_proof52111} as follows. For $k\geq 1$ and $n$ large enough,
\begin{align*}
&\hspace{-0.5cm}\mathbb{P}\left\{0<\frac{1}{Z_{S_{k-1}}}\sum_{j=1}^{Z_{S_{k-1}}}\1_{\{\widetilde{Z}^{(j)}_{d_k}(S_{k-1})>0\}} < \varepsilon n^{-3},S_{k-1}\leq n, Z_{S_{k-1}}>0\right\}\\
&=\mathbb{P}\left\{0<\frac{1}{Z_{S_{k-1}}}\sum_{j=1}^{Z_{S_{k-1}}}\1_{\{\widetilde{Z}^{(j)}_{d_k}(S_{k-1})>0\}} < \varepsilon n^{-3} \,\Big|\, S_{k-1}\leq n, Z_{S_{k-1}}>0\right\}\mathbb{P}\{S_{k-1}\leq n, Z_{S_{k-1}}>0\}\\
&\leq \frac{\widetilde{p}(1-\widetilde{p})\varepsilon}{(\widetilde{p}-\varepsilon n^{-3})}n^{-3}\mathbb{P}\{S_{k-1}\leq n, Z_{S_{k-1}}>0\}\leq \frac{\widetilde{p}(1-\widetilde{p})\varepsilon}{(\widetilde{p}-\varepsilon n^{-3})}n^{-3}\mathbb{P}\{S_{k-1}\leq n\}.
\end{align*}
Summarizing, the probability on the left-hand side of~\eqref{eq:flt_proof5211} is bounded from above by
$$
C_6 
n^{1/2}\frac{\widetilde{p}(1-\widetilde{p})\varepsilon}{(\widetilde{p}-\varepsilon n^{-3})}n^{-3}\sum_{k\geq 1}\mathbb{P}\{S_{k-1}\leq n\}=O(n^{-3/2}),\quad n\to\infty,
$$
thereby finishing the proof of~\eqref{eq:flt_proof5211} and 
Theorem~\ref{thm:flt}.

\bigskip

\noindent {\bf Acknowledgement} The research was supported by the High Level Talent Project DL2022174005L of Ministry of Science and Technology of PRC.

\bibliographystyle{plain}
\bibliography{BuraDongIksMar130323}

\end{document}